\documentclass[letterpaper,10pt]{amsart}

\usepackage{amsmath,amsthm,amssymb}
\usepackage{xspace,xcolor}
\usepackage[breaklinks,
            colorlinks,
            citecolor=teal,
            linkcolor=teal,
            urlcolor=teal,
            pagebackref,
            hyperindex]{hyperref}
\usepackage[alphabetic]{amsrefs}
\usepackage{color}
\usepackage{comment} 

\usepackage{tikz-cd}
\usepackage{mathrsfs}

\theoremstyle{plain}
\newtheorem{thm}{Theorem}[section]

\newtheorem{lem}[thm]{Lemma}

\theoremstyle{definition}

\newtheorem{question}[thm]{Question}

\theoremstyle{remark}
\newtheorem{rmk}[thm]{Remark}
\newtheorem{set}[thm]{Setting}

\def\Mustata{Mus\-ta\-\c{t}\u{a}\xspace}

\def\Z{{\mathbf Z}}
\def\Q{{\mathbf Q}}
\def\R{{\mathbf R}}
\def\C{{\mathbf C}}

\def\A{{\mathbf A}}
\def\P{{\mathbf P}}

\def\cD{\mathscr{D}}
\def\cE{\mathscr{E}}
\def\cF{\mathscr{F}}

\def\cH{\mathcal{H}}

\def\cM{\mathcal{M}}
\def\cN{\mathcal{N}}
\def\cO{\mathscr{O}}

\def\bff{{\bf f}}

\def\bD{{\bf D}}
\def\bR{{\bf R}}

\def\.{\cdot}
\def\^{\widehat}

\def\({\left(}
\def\){\right)}

\renewcommand{\and}{ \ \ \text{ and } \ \ }

\def\reg{\mathrm{reg}}
\def\sing{\mathrm{sing}}

\def\gr{{\mathrm{Gr}}}
\def\DR{\mathrm{DR}}

\DeclareMathOperator{\coker} {coker}

\DeclareMathOperator{\id} {id}

\def\cHom{{\cH}om}

\makeatletter
\@namedef{subjclassname@2020}{\textup{2020} Mathematics Subject Classification}
\makeatother

\usepackage{soul}

\begin{document}

\author[Q.~Chen]{Qianyu Chen}

\address{Department of Mathematics, University of Michigan, 530 Church Street, Ann Arbor, MI 48109, USA}

\email{qyc@umich.edu}


\begin{abstract}
We prove that the minimal exponent for local complete intersections satisfies an Inversion-of-Adjunction property. As a result, we also obtain the Inversion of Adjunction for higher Du Bois and higher rational singularities for local complete intersections.
\end{abstract}

\title[Inversion of Adjunction for the minimal exponent] {Inversion of Adjunction for the minimal exponent}

\maketitle

\section{Introduction}
\subsection{Main results}

The \emph{minimal exponent} $\widetilde\alpha(X,D)$ of a hypersurface $D$ in an irreducible and smooth complex algebraic variety $X$ was introduced by Saito in~\cite{Saito_microlocal} as the negative of the largest root of the reduced Bernstein-Sato polynomial $\widetilde b_D(s)$. It is a refinement of the ubiquitous invariant, \emph{log canonical threshold}, of the pair $(X,D)$, which is equal to $\min\{\widetilde\alpha(X,D),1\}$ by results of~\cites{Kollar,Lichtin}. When $D$ has isolated singularities, the minimal exponent has the name \emph{Arnold exponent} or \emph{complex singularity index}, studied in~\cites{Varchenko,Loe84,Steenbrink}. Some important features, such as the Restriction Theorem and semicontinuity for the log canonical threshold were extended into the minimal exponent via the theory of Hodge ideals~\cites{hi,hiq} developed by \Mustata and Popa. Recent generalizations to local complete intersections have been made in \cite{CDMO}, where the minimal exponent $\widetilde\alpha(X,Z)$ for a local complete intersection closed subscheme $Z$ was introduced using Hodge filtration and $V$-filtration and the same nice properties for the minimal exponent were established. Similar to the case of hypersurfaces, the minimal exponent $\widetilde\alpha(X,Z)$ refines the log canonical threshold of the pair $(X,Z)$ as showed in~\cite{BMS} and it can be characterized by the reduced Bernstein-Sato polynomial $\widetilde b_Z(s)$ defined in~\cite{BMS}, recently proved in~\cite{Brad}. We refer to Section~\ref{subsec:me} for the definition of the minimal exponent and review of its properties.

In this paper, we prove the following Inversion of Adjunction for the minimal exponent:

\begin{thm}\label{thm:inme}
Let $X$ be an irreducible and smooth complex algebraic variety and $Z$ be a local complete intersection closed subscheme in $X$. If there is a hypersurface $H$ in $X$ containing no irreducible component of $Z$ and a positive rational number $c$ such that,
\[
    \widetilde\alpha(X\setminus H,Z\setminus H)>c \quad  \text{and} \quad  \widetilde\alpha(X,Z\cap H)\geq c+1, 
\]
then we have $\widetilde\alpha(X,Z)>c$.
\end{thm}

Theorem~\ref{thm:inme} can be applied to the setting where $Z$ has a unique isolated singular point at $P\in Z$ and $H$ is a hypersurface in $X$ containing $P$. In this case, we deduce that $\widetilde\alpha(X,Z)> \widetilde\alpha(X,Z\cap H)+1$. If we further assume that $Z$ is a hypersurface in $X$ and $H$ is smooth, an improvement was obtained in~\cite{Loe84} and~\cite{DM22}.

The minimal exponent $\widetilde{\alpha}(X,Z)$ is closely related to higher Du Bois and higher rational singularities. These singularities have been recently studied as a generalization of classical Du Bois and rational singularities~\cites{MOPW, Saito_et_al, FL1,FL2,FL3,MP1,MP2,CDMO,CDM,shen2023k}. It has been shown that if the local complete intersection $Z$ has pure codimension $r$ in $X$, then $\widetilde{\alpha}(X,Z)\geq k+r$ (resp. $\widetilde{\alpha}(X,Z)> k+r$) is equivalent to that $Z$ having at worst $k$-Du Bois singularities (resp. $k$-rational singularities) in~\cites{MP1,CDMO,CDM}. We refer to Section~\ref{subsec_hbhr} for the precise definitions of higher Du Bois and higher rational singularities.

As a consequence of Theorem~\ref{thm:inme}, we obtain Inversion of Adjunction for $k$-Du Bois singularities and $k$-rational singularities for local complete intersections:

\begin{thm}\label{thm:inratbd}
Let $Z$ be a complex algebraic variety with local complete intersection singularities. If there is an effective Cartier divisor $D$ in $Z$ and a non-negative integer $k$ such that $Z\setminus D$ has $k$-rational singularities and that $D$ has $k$-Du Bois singularities then $Z$ has $k$-rational singularities.
\end{thm}

Inversion of Adjunction for Du Bois and rational singularities was proved by Schwede~\cite{Sch}*{Theorem 5.1} under the assumption that $Z\setminus D$ is smooth but no need that $Z$ is a local complete intersection. A more general statement on Inversion of Adjunction of Du Bois and rational pairs was proved in~\cite{KS} and~\cite{MSS}. An Hodge theoretic proof can be found in a recent paper~\cite{park}.

Recently higher Du Bois and higher rational singularities beyond local complete intersections were discussed in~\cite{shen2023k}. We ask

\begin{question}
Does Inversion of Adjunction for higher Du Bois and higher rational singularities hold without the local complete intersection assumption?
\end{question}

\subsection{Outline} We first review the related notions of the minimal exponent and the tools from mixed Hodge modules in Section~\ref{sec:bg}. We reduce the proof of Theorem~\ref{thm:inme} to a special form (Theorem~\ref{thm:strinratbd}) of Theorem~\ref{thm:inratbd} using several properties of the minimal exponent proved in~\cite{CDMO} in Section~\ref{subsec:red}. Lastly, we make use of mixed Hodge modules with ideas in~\cite{Sch} to prove Theorem~\ref{thm:strinratbd} through a new characterization (Lemma~\ref{lem:red}) of higher rational singularities.

\subsection{Acknowledgement} The author thanks Radu Laza for bringing attention to this question and he thanks Brad Dirks and Mircea \Mustata for valuable discussions. He is indebted to Mihnea Popa for several useful comments. He is very grateful to Mircea \Mustata for teaching him the argument in Section~\ref{subsec:red}. The author was partially supported by NSF grant DMS-1952399.

\section{Background} \label{sec:bg}

\subsection{Mixed Hodge modules}
We briefly recall some facts concerning mixed Hodge modules for the reader's convenience and lay down some notation which will be used. We refer to Saito's original papers~\cites{Saito_MHP,Saito_MHM} for details. A good reference for $\cD$-modules is~\cite{HTT}. 

We will mainly work with left $\cD$-modules. Let $X$ be a smooth complex algebraic variety of dimension $n$. A typical example of a pure Hodge module is the constant Hodge module $\Q_X^H[n]$, whose filtered $\cD_X$-module is $\cO_X$ with the Hodge filtration such that $\gr^F_p\cO_X=0$ for any $p\neq 0$. Sometimes, we will abuse the mixed Hodge module with the its underlying $\cD$-module. For a filtered left $\cD$-module $(\cM,F)$ underlying a mixed Hodge module, the associated graded de Rham complex of $(\cM,F)$ is
\begin{equation}\label{eq:drconv}
    \gr^F_{p}\DR_X(\cM) :=\big[ \underbrace{\gr^F_p\cM}_{\text{degree } -n} \to \underbrace{\Omega^1_X\otimes \gr^F_{p+1}\cM}_{\text{degree } -n+1} \cdots \to  \underbrace{\omega_X\otimes \gr^F_{p+n}\cM}_{\text{degree }0} \big]
\end{equation}
is a complex of coherent $\cO_X$-modules. For example, $\gr^F_p\DR_X(\cO_X)=\Omega_X^p[n-p]$. The definition can be easily carried over to the derived category of mixed Hodge modules $\rm D^b(\mathrm{MHM}(X))$. 

The six-functor formalism for mixed Hodge modules~\cite{Saito_MHM}*{Theorem 0.1}, extending the same formalism for the perverse sheaves, established by Saito, is crucial in this paper. We will frequently use adjunction and duality, which will be briefly reviewed. 

Denote by $\mathbf D_X$ the duality functor in $\rm D^b(\mathrm{MHM}(X))$. A polarization on a Hodge module $M$ of weight $w$, induces an isomorphism $M\cong \bD_X(M)(w)$. Here, $(w)$ denotes the Tate twist: on the level of filtered $\cD_X$-module $(\cM,F)$, by definition $(\cM,F_\bullet)(w)=(\cM,F_{\bullet-w})$. The functor $\bD_X$ is compatible with Grothendieck duality in the sense that~\cite{Saito_MHP}*{2.4.3}:  
\[
    \gr^F_p \DR_X (\mathbf D_X(-)) \cong \mathbf R \mathcal H om_{\cO_X}(\gr^F_{-p}\DR_X(-),\omega_X[n])
\]
for every $p\in\Z$ as functors from $\rm D^b(\mathrm{MHM}(X))$ to $\rm D^b_{\rm coh}(X)$. If it is clear from the context, we will also denote by $\bD_X$ the Grothendieck duality $\bR\cHom(-,\omega_X[n])$. Then the above becomes
\begin{equation}\label{eq:dualprop}
    \gr^F_p \DR_X \circ \mathbf D_X \cong \bD_X \circ \gr^F_{-p}\DR_X.
\end{equation}

For a morphism $f\colon X\to Y$ between smooth complex algebraic varieties, we will use $f_*\colon \rm D^b(\mathrm{MHM}(X))\to \rm D^b(\mathrm{MHM}(Y))$ to denote the direct image functor of mixed Hodge modules. Sometimes we will abuse it with the direct image functor $f_+\colon \rm D^b_{\rm rh}(\cD_X)\to \rm D^b_{\rm rh}(\cD_Y)$ on their regular holonomic $\cD$-modules. The functor $f_*$ is the right adjoint of the inverse image functor $f^*$. Denote by $f_!:=\bD_Y \circ f_*\circ \bD_X$ the proper direct image functor; it coincides with $f_*$ when $f$ is proper. Lastly, we have the proper inverse image functor $f^!:=\bD_X \circ f^* \circ \bD_Y$ which is the right adjoint of $f_!$. These functors are all compatible with the underlying $\Q$-complexes. When $f$ is proper, for every $p\in \Z$, we have a natural isomorphism between functors
\begin{equation}\label{eq:direct}
    \bR f_*\circ\gr^F_{p}\DR_X\cong \gr^F_{p}\DR_Y\circ f_+
\end{equation}
where $\bR f_*\colon \rm D^b_{\rm coh}(X)\to \rm D^b_{\rm coh}(Y)$ is the derived direct image functor, as functors from $\rm D^b(\mathrm{MHM}(X))$ to $\rm D^b_{\rm coh}(\cO_Y)$ by~\cite{Saito_MHP}*{2.3.7}.

We end this part by recalling the excision distinguished triangles~\cite{Saito_MHM}*{4.4.1}. For any closed immersion $i\colon Z \to X$ from a closed subvariety and the open immersion $j\colon X\setminus Z \to X$, we have the distinguished triangles 
\[
j_!j^!  \to \id \to i_*i^*  \xrightarrow{+1} \quad \text{and}\quad
i_!i^!  \to \id \to j_*j^* \xrightarrow{+1}
\]
in $\rm D^b(\mathrm{MHM}(X))$. In this case, $j^!$ and $j^*$ are just the restriction to the open subset $X\setminus Z$. If it is clear from the context, we use $j_!$ and $j_*$ as shorthand for $j_!j^!$ and $j_*j^*$, respectively.

\subsection{The minimal exponent}\label{subsec:me}
Suppose that $X$ is an irreducible and smooth complex algebraic variety and $f_1,\ldots,f_r\in\cO_X(X)$ are non-zero regular functions which define a closed subscheme $Z$ of $X$. Let $\bff:=(f_1,f_2\ldots,f_r)$ and let
\[
\iota\colon X\hookrightarrow X\times\A^r,\quad \iota(x)=\big(x,f_1(x),\ldots,f_r(x)\big)
\]
be the graph embedding along $\bff$. Let $B_{\bff}=\iota_+\cO_X$ be the $\cD$-module direct image of $\iota$. If $t_1,\ldots,t_r$ denote the standard coordinates on 
$\A^r$, then we can write 
\[
B_{\bff}=\bigoplus_{\alpha\in \Z_{\geq 0}^r}\cO_X\cdot\partial_t^{\alpha}\delta_{\bff},
\]
where $\partial_t^{\alpha}=\partial_{t_1}^{\alpha_1}\cdots \partial_{t_r}^{\alpha_r}$ for $\alpha=(\alpha_1,\ldots,\alpha_r)$ with the natural action of $\cO_X$ and of $\partial_{t_i}$. The actions of a vector field $\xi$ on $X$ and of the $t_i$ are given by 
\[
\xi\cdot h\partial_t^{\alpha}\delta_{\bff}=\xi(h)\partial_t^{\alpha}\delta_{\bff}-\sum_{i=1}^r\xi(f_i)h\partial_t^{\alpha+e_i}\delta_{\bff}
\]
and
\[
t_i\cdot h\partial_t^{\alpha}\delta_{\bff}=f_ih\partial_t^{\alpha}\delta_{\bff}-\alpha_ih\partial_t^{\alpha-e_i}\delta_{\bff},
\]
where $e_1,\ldots,e_r$ is the standard basis of $\Z^r$. In fact, $B_{\bff}$ underlies the pure Hodge module $\iota_*\Q_X^H[\dim X]$,
with the Hodge filtration $(F_p B_\bff)_{p\in\Z}$ given by
\[
F_{p}B_{\bff}=\bigoplus_{|\alpha|\leq p-r}\cO_X\cdot \partial_t^{\alpha}\delta_{\bff},
\]
where $|\alpha|=\alpha_1+\ldots+\alpha_r$. Note that $F_pB_\bff=0$ if $p<r$.

Let $(V^{\lambda}B_{\bff})_{\lambda\in\Q}$ be the $V$-filtration along $X\times \{0\}\subset X\times \A^r$, introduced by Kashiwara \cite{Kashiwara} and Malgrange \cite{Malgrange}. We refer to~\cites{BMS,CD} for the definition and properties of $V$-filtrations along a subvariety of any codimension.
The \emph{minimal exponent} $\widetilde{\alpha}(X,Z)$ defined~\cite{CDMO} is the unique positive rational number or $\infty$ determined by the following condition: for some positive integer $q$ and rational number $\alpha$ in $[0,1)$,
\begin{equation}\label{eq3_intro}
\widetilde{\alpha}(X,Z)
\geq q-\alpha \Longleftrightarrow \left\{
\begin{array}{cl}
F_rB_{\mathbf f}\subseteq V^{q-\alpha}B_{\bff} , & \text{if}\,\,\delta_{\mathbf f}\not\in V^rB_{\mathbf f}; \\[2mm]
 F_{q}B_{\bff}\subseteq V^{r-\alpha}B_{\bff} , & \text{if}\,\,\delta_{\mathbf f}\in V^rB_{\mathbf f},
\end{array}\right.
\end{equation}
Note that the convention for Hodge filtration in~\cite{CDMO} is non-standard: what we denote by $F_{q}B_{\bff}$ here is denoted by $F_{q-r}B_{\bff}$  in \cite{CDMO}. If $Z$ is define by a single regular function $f$, then the minimal exponent $\widetilde\alpha(X,Z)$ is also denote by $\widetilde\alpha(f)$.

Alternatively, the minimal exponent $\widetilde\alpha(X,Z)$ can be defined as the negative of the largest root of the \emph{reduced Bernstein-Sato polynomial} $\widetilde{b}_{\bff}(s):=b_{\bff}(s)/(s+r)$, where $b_{\bff}(s)$ is the Bernstein-Sato polynomial of ${\bff}$; see~\cite{BMS}. The fact that this agrees with the characterization~\eqref{eq3_intro} is a consequence of \cite{Saito-MLCT} and \cite{Brad}. 

\begin{rmk}\label{rem:ind}
    The minimal exponent $\widetilde\alpha(X,Z)$ depends on the embedding $i\colon Z\to X$ in a predicted way as pointed out by~\cite{CDMO}*{Proposition 4.14}: if $Z$ is embedded in another irreducible and smooth complex algebraic variety $Y$, then 
    \[
    \widetilde\alpha(X,Z)-\dim X=\widetilde\alpha(Y,Z)-\dim Y.
    \]
\end{rmk}

In the global setting, if $Z$ is a local complete intersection closed subscheme of pure codimension $r$, we can cover $X$ by by affine open subsets $U_1, U_2,\dots, U_N$, and put
\[
    \widetilde{\alpha}(X,Z):=\min_{i;Z\cap U_i\neq\emptyset}\widetilde{\alpha}(U_i,Z\cap U_i).
\]
As showed in~\cite{BMS}, the log canonical threshold ${\rm lct}(X,Z)$ is $\min\{\widetilde\alpha(X,Z),r\}$. 

It is also convenient for us to use the local version of the minimal exponent: for any point $x\in Z$, define
\begin{equation}\label{eq:locme}
\widetilde\alpha_x(X,Z):=\max \widetilde\alpha(U,Z\cap U)
\end{equation}
where the maximum runs over any open neighborhoods $U$ of the point $x$. The fact that it can achieve the maximum not only supremum is pointed in~\cite{CDMO}*{Definition 4.16}.

The following is a rephrasing of the main Theorem in~\cite{CDMO}.
 
\begin{thm} \label{thm:cdmo}
    Let $Z$ be a closed subscheme of an irreducible and smooth complex algebraic variety $X$ defined by a regular sequence $(f_1,f_2,\dots,f_r)$, let $[y_1:y_2:\dots:y_r]$ be a system of homogeneous coordinates on $\mathbf P^{r-1}$ and let $Z'$ be the hypersurface in $X':=X\times \mathbf{P}^{r-1}$ defined by the function  $g=y_1f_1+y_2f_2+\cdots y_rf_r$. Then we have $\widetilde\alpha(X', Z')=\widetilde\alpha(X,Z)$.
\end{thm}

This rephrasing has already been used in~\cite{CDMO} to study the Restriction Theorem and the semicontinuity for the minimal exponent~\cite{CDMO}*{Theorem 1.2i) and ii)}:

\begin{thm}\label{thm3_intro}
Let $X$ be an irreducible and smooth complex algebraic variety and 
let $Z$ be a local complete intersection closed subscheme of $X$ of pure codimension $r$. 
\begin{enumerate}
\item[i)] If $H$ is a smooth hypersurface in $X$ that contains no irreducible component of $Z$, then for every 
$x\in Z\cap H$, we have
\[
\widetilde{\alpha}_x(H, Z\cap H)\leq\widetilde{\alpha}_x(X, Z).
\]
\item[ii)] Given a smooth morphism $\mu\colon X\to T$ such that for every $t\in T$, $Z_t:=Z\cap \mu^{-1}(t)\hookrightarrow
X_t=\mu^{-1}(t)$ has pure codimension $r$, then the following hold: 
\begin{enumerate}
\item[${\rm ii_1)}$] For every $\alpha\in\Q_{>0}$, the set 
\[
\big\{x\in Z\mid \widetilde{\alpha}_x(X_{\mu(x)}, Z_{\mu(x)})\geq \alpha\}
\]
is open in $Z$. 
\item[${\rm ii_2)}$] There is an open subset $T_0$ of $T$ such that for every $t\in T_0$ and $x\in Z_t$, we have
\[
\widetilde{\alpha}_x(X_t, Z_t)=\widetilde{\alpha}_x(X, Z).
\]
\end{enumerate}
In particular, the set $\big\{\widetilde{\alpha}_x(X_{\mu(x)}, Z_{\mu(x)})\mid x\in Z\big\}$ is finite. Moreover, 
if $s\colon T\to X$ is a section of $\mu$ such that $s(T)\subseteq Z$, then the set 
$\big\{t\in T\mid \widetilde{\alpha}_{s(t)}(X_t, Z_t)\geq\alpha\big\}$ is open in $T$ for every $\alpha\in\Q_{>0}$.
\end{enumerate}
\end{thm}

\subsection{Higher Du Bois and higher rational singularities}\label{subsec_hbhr}
Let $Z$ be a complex algebraic variety and let $(\underline{\Omega}_Z^{\bullet},F)$ be the Du Bois complex of $Z$, introduced in \cite{DuBois} using hyperresolutions. We list some properties of the Du Bois complex: 
\begin{enumerate}
    \item  $\underline{\Omega}_Z^{\bullet}$ is a resolution of the constant sheaf $\C_Z$;
    \item  $\underline{\Omega}_Z^p:={\rm Gr}_F^p(\underline{\Omega}_Z^{\bullet})[p]$ is a complex of coherent sheaves and is acyclic unless $0\leq p\leq \dim Z$;
    \item there is a natural filtered morphism from the de Rham complex $\Omega_Z^{\bullet}$, with the ``stupid" filtration to $(\underline{\Omega}_Z^{\bullet},F)$; 
    \item the filtered morphism in (c) is filtered isomorphism over the smooth locus of $Z$.
\end{enumerate}    
The Du Bois complex plays a fundamental role in the Hodge theory for singular varieties; see \cite{PetersSteenbrink}*{Chapter~7.3}. 
We say that $Z$ has (at worst) \emph{$k$-Du Bois singularities}, following \cite{Saito_et_al}, if the natural morphism in the bounded derived category of coherent sheaves on $Z$
\[
\Omega_Z^p\to \underline{\Omega}_Z^p.
\] 
is an isomorphism for $0\leq p\leq k$. Clearly, $0$-Du Bois singularities are the same as Du Bois singularities.

We say a proper morphism $\mu\colon\widetilde{Z}\to Z$ from a smooth variety is a \emph{strong log resolution} if $\mu$ is isomorphic over $Z_\reg:=Z\smallsetminus Z_{\rm sing}$ and, $E=\mu^{-1}(Z_{\rm sing})$ is a simple normal crossing divisor. For a non-negative integer $k$, we say that $Z$ has (at worst) 
\emph{$k$-rational} singularities, following  \cite{FL1}, if there exists one (hence for any; see~\cite{MP2}*{Lemma~1.6}) strong log resolution $\mu\colon\widetilde{Z}\to Z$, such that the canonical morphism
\begin{equation}\label{eq_def_k_rat}
\Omega_Z^p\to \R\mu_*\Omega^p_{\widetilde{Z}}({\rm log}\,E)
\end{equation}
is an isomorphism for all $p\leq k$. The $0$-rational singularities are the same as the classical notion of rational singularities. 

In summary, we have the following theorem relating the minimal exponent and higher Du Bois and higher rational singularities:

\begin{thm}\label{thm:medbrat}
Let $X$ be an irreducible and smooth complex algebraic variety and $Z$ be a local complete intersection closed subscheme in $X$ of pure codimension $r$. Then, for any non-negative integer $k$,
\begin{enumerate}
    \item $\widetilde\alpha(X,Z)\geq k+r$ if and only if $Z$ has $k$-Du Bois singularities.
    \item $\widetilde\alpha(X,Z)> k+r$ if and only if $Z$ has $k$-rational singularities.
\end{enumerate}
In particular, if $Z$ has $k$-rational singularities then $Z$ has $k$-Du Bois singularities; and if $Z$ has $k$-Du Bois singularities then $Z$ has $(k-1)$-singularities.
\end{thm}
The case when $Z$ is a hypersurface was treated in~\cites{MOPW,Saito_et_al,FL1, FL2,MP2} and the case for local complete intersection was obtained in~\cites{MP1, CDMO, CDM} 

We recall a useful bound of the codimension of the singular locus for higher Du Bois and higher rational singularities~\cite{MP1}*{Corollary~3.40 and Theorem~F} and \cite{CDM}*{Corollary 1.3}.

\begin{thm}\label{sing_locus_Du_Bois}
Let $Z$ be a local complete intersection variety. If $Z$ has $k$-Du Bois singularities, then ${\rm codim}_Z(Z_{\rm sing})\geq 2k+1$; if $Z$ has $k$-rational singularities then ${\rm codim}_Z(Z_{\rm sing}) \geq 2k+2$.
\end{thm}

Saito's theory of mixed Hodge modules~\cites{Saito_MHP,Saito_MHM} is a convenient tool to study higher Du Bois and higher rational singularities. Let $X$ be an irreducible and smooth $n$-dimensional complex algebraic variety and $Z$ be a pure $d$-dimensional closed subscheme of $X$. Let $i\colon Z\hookrightarrow X$ be the closed inclusion and $r=n-d$ be the codimension of $Z$ in $X$. Let $\Q_Z^H:=a_Z^*\Q^H_{\star}$, as an object in $\rm D^b(\mathrm{MHM}(Z))$ where $a_Z\colon Z\to \star$ is the morphism to a point and $\Q^H_{\star}$ is the Hodge structure of weight $0$ on the field $\Q$. It is a consequence of 
\cite{Saito-HC}*{Theorem~0.2} (see also~\cite{MP1}*{Proposition 5.5} for a simpler proof) that for every $p\in \Z$, we have a natural isomorphism
\begin{equation}\label{eq_Saito_DuBois}
\underline{\Omega}_Z^p[d-p]\cong {\rm Gr}^F_{-p}{\rm DR}_X(i_*\Q_Z^H[d])
\end{equation}
in ${\rm D}^b_{\rm coh}(X)$. In general, $i_*\Q_Z^H[d]=i_*i^*\Q_X^H[d]$ is an object in $\rm D^b(\mathrm{MHM}(X))$ but when $Z$ is a local complete intersection, we know that $i_*\Q_Z^H[d]$ is a mixed Hodge module as $\Q_Z[d]$ is a perverse sheaf. Because of~\eqref{eq_Saito_DuBois}, the scheme $Z$ has $k$-Du Bois singularities if and only if the induced morphism 
\begin{equation}\label{eq:defdb}
\Omega_Z^p[d-p] \to  \underbrace{{\rm Gr}^F_{-p}{\rm DR}_X(i_*\Q_Z^H[d])}_{\cong \underline{\Omega}_Z^p[d-p]}
\end{equation}
is a quasi-isomorphism for every $p\leq k$. 

The log de Rham complex is also related to mixed Hodge modules. Let $\mu\colon \widetilde Z\to Z$ be a strong log resolution such that $E=\mu^{-1}(Z_{\rm sing})$ is a divisor with simple normal crossing support. By the result of~\cite{deligne06} and~\cite{Saito_MHM}*{3.11}, we have
\[
\Omega_{\widetilde Z}^k(\log E)[d-k] \xrightarrow{\cong} \gr^F_{-k}\DR_{\widetilde Z}\big(\cO_{\widetilde Z}(*E)\big).
\]
Let $j'\colon  X\setminus Z_{\sing}\to X$ be the open immersion.  To simplify the notation, denote by $i_*\Q_{Z_{\rm reg}}^H[d]$ the object $j'_*j'^*i_* \Q^H_{Z}[d]$ in the derived category of mixed Hodge modules. There is a natural morphism induced by the adjunction $\id \to j'_*j'^*$ in $\rm D^b({\rm MHM}(X))$,
\begin{equation}\label{eq:singext}
    i_*\Q_Z^H[d]\to i_*\Q_{Z_{\reg}}^H[d].
\end{equation}
Putting $\widetilde j \colon \widetilde Z\setminus E \to \widetilde Z$ for the open immersion, because of 
\[
i_* \Q^H_{Z_{\reg}}[d]=i_*\mu_*\widetilde j_*\Q^H_{\widetilde Z\setminus E}[d]
\]
it follows from~\eqref{eq:direct} that
\begin{equation}\label{eq:logres}
    \begin{aligned}
    &\bR \mu_* \Omega_{\widetilde Z}^k(\log E)[d-k] \cong \bR\mu_*\gr^F_{-k}\DR_{\widetilde Z}\big(\cO_{\widetilde Z}(*E)\big) \\
\cong &\gr^F_{-k}\DR_X\left(\mu_+\cO_{\widetilde Z}(*E)\right) \cong \gr^F_{-k}\DR_X\left(i_*\Q_{Z_{\rm reg}}^H[d]\right)
\end{aligned}
\end{equation}
as $\cO_{\widetilde Z}(*E)$ is the filtered $\cD$-module of the underlying mixed Hodge module $\widetilde j_*\Q^H_{\widetilde Z\setminus E}[d]$. 
Hence, the scheme $Z$ has $k$-rational singularities if and only if 
\[
\Omega^k_Z[d-k] \to \gr^F_{-k}\DR_X\left(i_*\Q_{Z_{\rm reg}}^H[d]\right)
\]
is quasi-isomorphic for $p\leq k$. Deriving from Theorem~\ref{thm:medbrat}, the local complete intersection variety $Z$ has $k$-rational singularities if and only if, $Z$ has $k$-Du Bois singularities and the natural morphism
\begin{equation}\label{eq:def}
\underbrace{\gr^F_{-k}\DR_X\left(i_*\Q_Z^H[d]\right)}_{\cong\underline{\Omega}_Z^k[d-k]} \to \underbrace{\gr^F_{-k}\DR_X\left(i_*\Q_{Z_{\rm reg}}^H[d]\right)}_{\cong\mathbf R\mu_*\Omega^k_{\widetilde Z}(\log E)[d-k]}
\end{equation}
is a quasi-isomorphism. 

\subsection{A preparation lemma} 

We conclude this section by the following criterion of $k$-rational singularities.

\begin{lem}\label{lem:red}
Let $X$ be an irreducible and smooth complex algebraic variety and $Z$ is a local complete intersection closed subscheme of $X$. The closed subscheme $Z$ has $k$-rational singularities if and only if it has $k$-Du Bois singularities and the canonical morphism induced by the dual of~\eqref{eq:def}
\begin{equation}\label{eq:prep}
    \cH^0\bD_X\left(\gr^F_{-k}\DR_X\left(i_*\Q^H_{Z_{\reg}}[d]\right)\right) \to \cH^0\bD_X\left(\gr^F_{-k}\DR_X\left(i_*\Q_{Z}^H[d]\right)\right)
\end{equation}
is a surjection.
\end{lem}

\begin{rmk}
    If we work out~\eqref{eq:prep} using~\eqref{eq:def} and the compatibility of the duality functors~\eqref{eq:dualprop}, assuming that $Z$ has pure dimension $d$, the lemma implies that $Z$ has $k$-rational singularities if and only if $Z$ has $k$-Du Bois singularities and there exists a strong log resolution $\mu\colon \widetilde Z\to Z$ such that the canonical morphism
    \[
         R^k \mu_* \Omega^{d-k}_{\widetilde Z}(\log E)(-E) \to \mathcal{E}xt_{\cO_Z}^k(\underline{\Omega}^k_Z,\omega_Z)
    \]
    is surjective, where $\omega_Z$ is the dualizing sheaf of $Z$.
\end{rmk}

Before giving the proof, we review a construction in~\cite{Saito_MHM}*{4.5.12}, which was used in~\cite{CDM} to study the $k$-rational singularities for local complete intersections. 

Let $Z$ be a local complete intersection closed subscheme of pure dimension $d$ in an irreducible and smooth complex algebraic variety $X$ of dimension $n$. Put $r=n-d$. Then $\Q_Z^H[d]$ has the top weight $d$, while $\cH^r_Z(\cO_X)$ has the lowest weight $n+r$ as constructed in~\cite{Saito_MHM}*{4.5.2}. Here, abusing the notation, we use the $\cD$-module $\cH^r_Z(\cO_X)$ to denote the mixed Hodge module $\cH^r(i_!i^!\Q_X^H[n])$ where $i\colon Z\to X$ is the closed immersion. Indeed, we have
\begin{equation}\label{eq:lcdual}
    \mathbf D_X(i_*\Q^H_Z[d])=i_!i^!\Q_X^H[n+r](n)=\cH^r_Z(\cO_X)(n),
\end{equation}
We also point out that $\gr^W_d i_* \Q^H_{Z}[d]$ is canonically isomorphic to the intersection Hodge module $i_*\mathrm{IC}_Z^H$.

The morphism
\begin{equation}\label{eq:tau}
    \begin{tikzcd}[column sep=small]
\tau\colon    i_* \Q^H_{Z}[d] \arrow[r, twoheadrightarrow, "\gamma_Z"] & \gr^W_d i_* \Q^H_{Z}[d] \arrow[r, "\cong"] & W_{n+r} \cH^r_Z(\cO_X)(r) \arrow[r, hook, "\gamma_Z^\vee"] & \cH^r_Z(\cO_X)(r),
\end{tikzcd}
\end{equation}
obtained by composing the surjection $\gamma_Z$, an isomorphism induced by any polarization 
\[
\gr^W_d i_* \Q^H_{Z}[d] \xrightarrow{\cong} \bD_X\big(\gr^W_d i_* \Q^H_{Z}[d]\big)(d)\cong \left(W_{n+r}\cH^r_Z(\cO_X)\right)(r),
\]
and $\gamma_Z^\vee:=\bD_X(\gamma_Z)(-d)$, plays an important role in this paper. Note that $\tau$ is isomorphic over the complement of the singular locus $X\setminus Z_{\sing}$ and $\bD_X(\tau)=\tau(d)$ by~\eqref{eq:lcdual}. 
%\begin{equation}\label{eq:taudual}
  %  \bD_X(\tau)=\tau(d)
%\end{equation}

Here are two basic facts from~\cite{CDM} regarding the morphism $\tau$.

\begin{lem}[\cite{CDM}*{Lemma 3.5}]\label{lem:dualfw}
If $F_pW_{n+r}\cH^r_Z(\cO_X)=F_p \cH^r_Z(\cO_X)$ for some $p\in \Z$, then the surjective map
\[
F_{p+r+1}\gamma_Z\colon F_{p+r+1}i_*\Q_Z^H[d] \to F_{p+r+1}\gr^W_d i_* \Q^H_{Z}[d]
\]
is an isomorphism; in particular, if $F_{p+r}\tau$ is an isomorphism then $F_{p+r+1}\tau$ is injective.
\end{lem}

It is a direct corollary of~\cite{CDM}*{Theorem 2.3, 2.5 and 3.1} that 

\begin{thm}\label{thm:3.1}
The closed subscheme $Z$ has $k$-rational singularities if and only if $Z$ has $k$-Du Bois singularities and 
\[
F_{k+r}\tau\colon F_{k+r}i_*\Q^H_Z[d]\to F_k \mathcal H^r_Z(\cO_X)
\] 
is an isomorphism.   
\end{thm}

\begin{proof}[Proof of Lemma~\ref{lem:red}]
The ``only if'' part is clear by the discussion in the end of Section~\ref{subsec_hbhr}. %part by taking the $0$-th cohomology and the dual of~\eqref{eq:def}. 

We prove the ``if'' part. Let $i'\colon Z_{\sing}\to Z$ be the closed immersion from the singular locus, $j\colon Z_{\reg} \to Z$ and $j'\colon X\setminus Z_\sing \to X$ be the open immersions. Because the morphism $\tau\colon i_*\Q_Z^H[d] \to \cH^r_Z(\cO_Z)(r)$, constructed as~\eqref{eq:tau}, is isomorphic over $X\setminus Z_\sing$, the natural map $i_*\Q^H_Z[d]\to i_*\Q^H_{Z_\reg}[d]$ is factored through $\tau$ due to the commutative diagram
\begin{equation}\label{eq_factor}
\begin{tikzcd}
    i_*\Q^H_Z[d] \arrow[r,"\tau"]\arrow[d] & \cH^r_Z(\cO_X)(r) \arrow[d]\arrow[dl,"\varepsilon"] \\
    i_*\Q^H_{Z_\reg}[d] \arrow[r, equal] & j'_*j'^*\cH^r_Z(\cO_X)(r)
\end{tikzcd}
\end{equation}
obtained by applying adjunction $\id \to j'_*j'^*$ to $\tau$, recalling that $i_*\Q^H_{Z_\reg}[d]:=j'_*j'^*i_*\Q^H_{Z}[d]$. Taking the dual, combined with~\eqref{eq:dualprop} gives the following commutative diagram:
\begin{equation}\label{eq:diared}
\begin{tikzcd}
    \gr^F_{k-d}\DR_X \left(i_*j_! \Q^H_{Z_\reg}\right) \arrow[r, equal] \arrow[d,"\cong \, \text{ by}~\eqref{eq:lrtrick}"] & \bD_X\left(\gr^F_{-k}\DR_X\left(i_*\Q^H_{Z_{\reg}}[d]\right)\right) \arrow[d, "\text{induced by }\varepsilon"] \\ 
     \gr^F_{k-d}\DR_X(i_*\Q_Z^H[d]) \arrow[r, equal] \arrow[d,"\gr^F_{k-d}\DR_X(\tau)"] & \bD_X\left(\gr^F_{-k-r}\DR_X\left(\cH^r_Z(\cO_X)\right)\right) \arrow[d,"\text{induced by }\tau"] \\
    \gr^F_{k-n}\DR_X \big(\cH^r_Z(\cO_X)\big) \arrow[r,equal] & \bD_X\left(\gr^F_{-k}\DR_X\left(i_*\Q_{Z}^H[d]\right)\right) 
\end{tikzcd}.
\end{equation}

The morphism $\eta=\bD_X(\varepsilon)[-d](-d)$, can be fitted into the distinguished triangle in $\rm D^b(\mathrm{MHM}(Z))$:
\begin{equation}\label{eq:fl}
  i_*j_!\Q^H_{Z_\reg} \xrightarrow{\eta} i_*\Q^H_Z \to i_*i'_*\Q_{Z_{\sing}}^H \xrightarrow{+1},   
\end{equation}
due to $i_*j_!\Q^H_{Z_\reg}[d](d)=\bD_X\left(i_*\Q^H_{Z_\reg}[d]\right)$ and $i_*\Q^H_Z[d](d)=\bD_X(\cH^r_Z(\cO_X)(r))$.

Because $Z$ has $k$-Du Bois singularities, $\dim Z_{\sing}\leq d- 2k-1<d-k$ thanks to Theorem~\ref{sing_locus_Du_Bois}. Hence, it follows from~\eqref{eq_Saito_DuBois} that $\gr^F_{k-d} \DR_X(i_* i'_*\Q^H_{Z_{\sing}}) {\cong} \underline{\Omega}^{d-k}_{Z_{\sing}}$ is acyclic. Then by~\eqref{eq:fl}, we have 
\begin{equation}\label{eq:lrtrick}
    \gr^F_{k-d}\DR_X(i_*j_! \Q^H_{Z_\reg}) \xrightarrow{\cong} \gr^F_{k-d}\DR_X(i_*\Q^H_Z).     
\end{equation}
This isomorphism has already been observed in~\cites{FL2,MP2}.

Therefore, via the diagram~\eqref{eq:diared}, the surjectivity of~\eqref{eq:prep} is equivalent to that
\begin{equation}\label{eq:h0}
\cH^0\gr^F_{k-d}\DR_X(\tau) \colon \cH^0\gr^F_{k-d}\DR_X(i_*\Q_Z^H[d]) \to \cH^0\gr^F_{k-n}\DR_X(\cH^r_Z\big(\cO_X)\big)
\end{equation}
is surjective.
 
 As $Z$ has $k$-Du Bois singularities (in particular, $(k-1)$-rational singularities), if we expand the morphism of complexes 
\begin{equation}\label{eq:cx}
\gr^F_{k-d}\DR_X(\tau) \colon \gr^F_{k-d}\DR_X\left(i_*\Q_Z^H[d]\right) \to \gr^F_{k-n}\DR_X\big(\cH^r_Z(\cO_X)\big)
\end{equation}
as the commutative diagram:
\[
\begin{tikzcd}[column sep=small]
    0 \arrow{r} & \gr^F_{k-d}(i_*\Q_Z^H[d]) \arrow[r]\arrow[d] & \cdots \arrow[r] & \omega_X\otimes \gr^F_{k+r}(i_*\Q_Z^H[d]) \arrow[d] \arrow{r} & 0 \\
   0\arrow{r} & \underbrace{\gr^F_{k-n}\big(\cH^r_Z(\cO_X)\big)}_{{\rm degree}\, -n} \arrow[r] & \cdots \arrow[r] &  \underbrace{\omega_X\otimes\gr^F_{k}\big(\cH^r_Z(\cO_X)\big)}_{{\rm degree}\, 0} \arrow{r} & 0
\end{tikzcd}
\]
we see that~\eqref{eq:cx} is isomorphism in cohomological degree $-n,-n+1,\dots,-1$ and, is injective in cohomological degree $0$ by Lemma~\ref{lem:dualfw} and Theorem~\ref{thm:3.1}. Therefore, the morphism~\eqref{eq:h0} is an isomorphism. Then an application of the $5$-lemma implies that~\eqref{eq:cx} is a term-wise isomorphism, which gives that $F_{k+r}\tau$ is an isomorphism because of~\cite{CDM}*{Lemma 2.1}. Hence, we conclude the proof by Theorem~\ref{thm:3.1}. 
\end{proof}

\section{Proof of Main Results}

\subsection{Some reductions}\label{subsec:red} 
We now proceed to prove Theorem~\ref{thm:inme}. The argument in this section is pointed out by Mircea \Mustata to the author. We are now in the following setting of Theorem~\ref{thm:inme}: 

\begin{set}\label{set}
    Assume that $X$ is an irreducible and smooth complex algebraic variety, $Z$ is a local complete intersection closed subscheme in $X$, $H$ is a hypersurface of $X$ containing no irreducible component of $Z$ and $c$ is a positive rational number such that 
\begin{equation}\label{eq:cond}
    \widetilde\alpha(X\setminus H,Z\setminus H)>c \quad  \text{and} \quad \widetilde\alpha(X,Z\cap H)\geq c+1.
\end{equation}
\end{set}

We first perform a useful reduction:

\begin{lem}\label{lem:reduction}
To prove Theorem~\ref{thm:inme}, it suffices to assume that the hypersurface $H$ is smooth.
\end{lem}

\begin{proof}
Suppose that we are in Setting~\ref{set} but the subset $\Sigma\subset Z$ consisting of the points $x\in Z$ such that $\widetilde\alpha_x(X,Z)\leq c$ is non-empty. By the definition~\eqref{eq:locme} of local minimal exponent, we know that $\Sigma$ is a closed subvariety. 

To achieve a contradiction, we can assume that $\Sigma$ has dimension $0$ because cutting down $(X,Z,H)$ by general hyperplanes passing through $\Sigma$ does not change $\widetilde\alpha_x(X,Z)$ for $x$ in the hyperplane sections according to Theorem~\ref{thm3_intro}(i). Shrinking $X$ if necessary, we can further assume that $\Sigma$ contains exactly one point $P$, $Z$ is cut out by a regular sequence $(f_1,f_2,\dots,f_r)$ and $H$ is defined by a regular function $h$. Let $W$ be a general smooth hyperplane, defined by a regular function $w$, passing through the point $P$. Consider the family of hyperplane sections of $Z$: $\mathcal{S}\subset X\times \A^1$ defined by the regular functions $f_1,f_2,\dots,f_r,t\cdot w+(1-t)\cdot h$ where $t$ is the coordinate on $\A^1$. Denote by the general fiber $\mathcal S_t=Z\cap H_t$ where $H_t$ is the hypersurface defined by $t\cdot w+(1-t)\cdot h$ in $X$. Note that the central fiber of $\mathcal S\to \A^1$ is $\mathcal S_0=Z\cap H$.

Applying Theorem~\ref{thm3_intro}(ii) to the second projection $X\times \A^1\to \A^1$ with the section $s\colon \A^1\to X\times \A^1, \, t\mapsto (P,t)$ and the local complete intersection closed subscheme $\mathcal S\subset X\times \A^1$, we deduce that
\[
\widetilde\alpha_P(X,\mathcal{S}_t) \geq \widetilde\alpha_P(X,\mathcal{S}_0)\geq c+1
\]
holds for any $t$ in a neighborhood $U\subset \A^1$ of $0$. Because $X\setminus H_t$ does not contain $P$,  the hypersurface $H_t$ is smooth and contains no irreducible component of $Z$ for general $t\in U$, if we replace $H$ by a general $H_t$ we are still in Setting~\ref{set}. %the tuple $(X,Z,H_t,c)$ satisfy the condition (setting~\ref{set}) in Theorem~\ref{thm:inme} for general $t$. 
Hence, if we have Theorem~\ref{thm:inme} for the tuple $(X,Z,H_t,c)$, then $\widetilde \alpha_P(X,Z)>c$, which is a contradiction. Thus, the closed subvariety $\Sigma$ is empty, i.e. $\widetilde\alpha(X,Z)>c$. 
\end{proof}

Theorem~\ref{thm:inme} can be further reduced to the following special case of Theorem~\ref{thm:inratbd}: 

\begin{thm}\label{thm:strinratbd}
Let $Z$ be a hypersurface of dimension $d$ in an irreducible and smooth complex algebraic variety $X$ of dimension $n=d+1$. If there is a smooth hypersurface $H$ in $X$ containing no irreducible component of $Z$ and a non-negative integer $k$ such that $Z\setminus H$ has $k$-rational singularities and that $Z\cap H$ has $k$-Du Bois singularities then $Z$ has $k$-rational singularities.
\end{thm}

\begin{proof}[Proof of Theorem~\ref{thm:inme} assuming Theorem~\ref{thm:strinratbd}]
Suppose we are in the Setting~\ref{set}. We can assume that $H$ is smooth by Lemma~\ref{lem:reduction}. Since the statement is local, we may and will assume that there is a regular sequence  $(f_1,f_2,\dots,f_r)$ cutting out $Z$. Let $X':=X\times \P^{r-1}$ and let $Z'$ be the hypersurface in $X'$ defined by 
\[
g=y_1f_1+y_2f_2+\cdots +y_rf_r
\]
where $[y_1:y_2:\dots:y_r]$ are homogeneous coordinates on $\P^{r-1}$.
Then by Theorem~\ref{thm:cdmo}, we have $\widetilde\alpha(X,Z)=\widetilde\alpha(X',Z')$. 

Denote by $H'$ the smooth hypersurface $H\times \P^{r-1}$ in $X'$. Clearly, $H'$ contains no irreducible component of $Z'$. By Theorem~\ref{thm:cdmo} again, we also have,
\[
\widetilde\alpha(X'\setminus H', Z'\setminus H')=\widetilde\alpha(X\setminus H,Z\setminus H)>c.
\] 
On the other hand, since $Z'\cap H'$ is defined by
\[
g|_{H'}=y_1f_1|_H+y_2f_2|_H+\cdots +y_rf_r|_H
\] 
Another application of Theorem~\ref{thm:cdmo} gives $\widetilde\alpha(H',Z'\cap H')=\widetilde\alpha(H,Z\cap H)$.
Then, as $H$ and $H'$ are smooth, together with Remark~\ref{rem:ind}
\[
\begin{aligned}
    &\widetilde\alpha(X',Z'\cap H') =\widetilde\alpha(H',Z'\cap H')+1 \\
=&\widetilde\alpha(H,Z\cap H)+1=\widetilde\alpha(X,Z\cap H)\geq c+1.
\end{aligned} 
\] 
Hence, replacing $(X,Z,H)$ by $(X',Z',H')$ we are still in the Setting~\ref{set}; together with $\widetilde\alpha(X,Z)=\widetilde\alpha(X',Z')$, we may and will assume that $Z$ is a hypersurface in $X$. 

If $c$ is a positive integer, writing $c=k+1$, Theorem~\ref{thm:medbrat} implies that $Z\setminus H$ has $k$-rational singularities and $Z\cap H$ has $k$-Du Bois singularities. Theorem~\ref{thm:strinratbd} shows that $Z$ has $k$-rational singularities and thus, $\widetilde\alpha(X,Z)>c$. 

Otherwise, assume that $Z$ is defined by a regular function $f$. Suppose that $\lceil c \rceil -c=\frac{m}{N}$ for some positive integers $m$ and $N$. Let $Z''$ be the hypersurface in $X'':=X\times \A^m$ defined by 
\[
f+ w_1^N+w_2^N +\cdots + w_m^N,
\] 
where $(w_1,w_2,\dots,w_m)$ is a system of coordinates on $\A^m$. Note that the singular locus $Z''_{\sing}$ of $Z''$ is exactly $Z_{\sing}\times\{0\}$; in particular, for $x\in Z$ and $y\in \A^m$, $\widetilde\alpha_{(x,y)}(X'',Z'')=\infty$ unless $x\in Z_{\sing}$ and $y=0$. Also, by the Thom-Sebastiani theorem for minimal exponents~\cite{Saito_microlocal}, combined with the fact that $\widetilde\alpha_0(w^N_i)=\frac{1}{N}$, we have
\begin{equation}\label{eq:ts}
\widetilde\alpha_{(x,0)}(X'',Z'')=\widetilde\alpha_x(X,Z)+\frac{m}{N},    
\end{equation}
holds for any $x\in Z_{\sing}$. Denote by $H''$ the hypersurface $H\times \A^m$ in $X''$ which clearly contains no irreducible component of $Z''$. It follows from the inequality
\[
 \widetilde\alpha_{(x,0)}(X'',Z'')=\widetilde\alpha_x(X,Z)+\frac{m}{N} > c+\frac{m}{N}=\lceil c\rceil,
\]
for any point $(x,0)$ in $Z''_\sing \setminus H''$ that $\widetilde\alpha(X''\setminus H'',Z''\setminus H'')> \lceil c\rceil$. On the other hand, since $Z''\cap H''$ is defined by
\[
f|_H+w_1^N+w_2^N +\cdots + w_m^N,
\]
whose singular locus is exactly $(Z\cap H)_{\sing}\times \{0\}$, we can apply the Thom-Sebastiani theorem again for $Z''\cap H''$:  
\[
\widetilde\alpha_{(x,0)}(H'',Z''\cap H'')=\widetilde\alpha_x(H,Z\cap H)+\frac{m}{N}
\]
for any $x\in (Z\cap H)_{\sing}$. Arguing as above we find that
\[
\widetilde\alpha(H'',Z''\cap H'') \geq c+\frac{m}{N}=\lceil c\rceil.
\] 
Theorem~\ref{thm:strinratbd}, combined with Theorem~\ref{thm:medbrat}, applying to $(X'',Z'',H'')$ and $\lceil c\rceil$ implies that $\widetilde\alpha(X'',Z'')> \lceil c\rceil$; in particular, for any singular point $x$ of $Z$, we have 
\[
\widetilde\alpha_x(X,Z)=\widetilde\alpha_{(x,0)}(X'',Z'')-\frac{m}{N}>\lceil c\rceil-\frac{m}{N}=c
\]
by~\eqref{eq:ts}, which completes the proof.
\end{proof}

\subsection{Proof of Theorem~\ref{thm:strinratbd}}
By the Restriction Theorem for the minimal exponent (Theorem~\ref{thm3_intro}(i)) and the assumption that $Z\setminus H$ is $k$-rational, the subscheme $Z$ has $k$-Du Bois singularities. It suffices to prove that the natural map
\[
\cH^0\bD_X\left(\gr^F_{-k}\DR_X\left(i_*\Q^H_{Z_{\reg}}[d]\right)\right) \to \cH^0\bD_X\left(\gr^F_{-k}\DR_X\left(i_*\Q_{Z}^H[d]\right)\right)
\]
is surjective, thanks to Lemma~\ref{lem:red}.

\subsubsection{} 
Let $T$ be the union of $H$ with the singular locus $Z_{\sing}$ of $Z$. The Cartesian diagram of open immersions of varieties
\[
\begin{tikzcd}
X\setminus T \arrow{r}\arrow{d}\arrow{dr}{j_T} &  X\setminus H \arrow{d}{j} \\
X\setminus Z_{\sing} \arrow{r}{j'} & X 
\end{tikzcd}
\]
induces the following commutative diagram in $\rm D^b(\mathrm{MHM}(X))$:
\[
\begin{tikzcd}
i_*\Q^H_{Z \setminus T}[d]  &  i_*\Q^H_{Z\setminus H}[d]  \arrow{l}{\alpha} \\
 i_*\Q^H_{Z_{\reg}}[d] \arrow{u} &  i_*\Q^H_Z[d] \arrow{u}{\beta} \arrow{l}
\end{tikzcd}
\]
Here, $i:Z\to X$ is the closed immersion and, abusing the notation, denote by $i_*\Q^H_{Z\setminus H}[d]$ the object $j_*j^*i_*\Q^H_Z[d]$, by $i_*\Q^H_{Z_{\reg}}[d]$ the object ${j'}_*j'^*i_*\Q^H_{Z}[d]$ and by $i_*\Q^H_{Z\setminus T}[d]$ the object ${j_T}_*j_T^*i_*\Q^H_{Z}[d]$ in $\rm D^b(\mathrm{MHM}(X))$. It suffices to prove that 
\[
\cH^0\bD_X\left(\gr^F_{-k}\DR_X(\alpha)\right) \quad \text{and} \quad \cH^0\bD_X\left(\gr^F_{-k}\DR_X(\beta)\right)
\]
are both surjections.

\subsubsection{} 
Using the compatibility between the two duality functors~\eqref{eq:dualprop}, we see that $\bD_X\left(\gr^F_{-k}\DR_X\left(\alpha\right)\right)$ is the same as 
\[
\gr^F_{k}\DR_X(\bD_X(\alpha)) \colon \gr^F_{k-d}\DR_X\left({j_{T}}_!i_*\Q^H_{Z\setminus T}[d]\right)\to \gr^F_{k-n} \DR_X\left(j_!\cH^1_Z(\cO_X)\right).
\] 
%where $j:X\setminus H\to Z$ and $j_T:X\setminus T\to X$ are the two open immersions. 
We shall prove that this is a quasi-isomorphism. For this, we make use of the morphism 
\[
\tau\colon i_*\Q^H_{Z}[d] \to \cH^1_Z(\cO_X)(1) 
\] 
constructed as~\eqref{eq:tau}. As $j_T^!\tau$ is identity, we have the factorization 
\[
\begin{tikzcd}[row sep=small]
    {j_{T}}_!i_*\Q^H_{Z\setminus T}[d] \arrow[r]\arrow[d, equal] & j_!i_*\Q_{Z\setminus H}^H[d] \arrow[r,"j_!\tau"] & j_!\cH^1_Z(\cO_X)(1)  \\
    {j_T}_! \cH^1_Z(\cO_X)(1) \arrow[rru, bend right=10]   &   &
\end{tikzcd}
\]
by applying the adjunction ${j_{T}}_!j_T^! \to \id$ to $j_!\tau$. Thus, this step can be concluded if we show that both $\gr^F_{k-d}\DR_X(j_!\tau)$ and 
\[
\gr^F_{k-d}\DR_X\left({j_{T}}_!i_*\Q^H_{Z\setminus T}[d]\right)\to\gr^F_{k-d}\DR_X\left({j}_! i_*\Q^H_{Z\setminus H}[d]\right) 
\] 
are quasi-isomorphisms. The following lemma takes care of the part of $\gr^F_{k-d}\DR_X(j_!\tau)$. 

\begin{lem}\label{lem:ext}
The canonical morphisms 
\[
F_{\ell+1} j_!\tau\colon F_{\ell+1}{j}_! i_*\Q^H_{Z\setminus H}[d] \xrightarrow{} F_{\ell} {j}_!\cH^1_Z(\cO_X)
\]
and 
\[
F_{\ell+1} j_*\tau\colon F_{\ell+1}{j}_* i_*\Q^H_{Z\setminus H}[d] \xrightarrow{} F_{\ell} {j}_*\cH^1_Z(\cO_X)
\]
are isomorphisms for $\ell\leq k$. In particular, $\gr^F_{k-d}\DR_X(j_!\tau)$ and $\gr^F_{k-d}\DR_X(j_*\tau)$ are both quasi-isomorphisms.
\end{lem}
\begin{proof}[Proof of the lemma]
Since $Z\setminus H$ has $k$-rational singularities, 
\[
F_{\ell+1}j^!\tau \colon F_{\ell+1}j^!i_*\Q^H_{Z}\to F_\ell j^!\cH^1_{Z}(\cO_{X})
\] 
are isomorphisms for $\ell \leq k$ by Theorem~\ref{thm:3.1}. We can assume that $H$ is defined by a regular function $t$ as the statement is local. For a mixed Hodge module $M$, by definition~\cite{Saito_MHP}*{3.2.2}:
\begin{equation}\label{eq:ext}
F_\ell (j_!M) = \sum_{i\geq 0} \partial^i_t \cdot F_{\ell-i} V^{1}(j_!M) \quad \text{and} \quad F_\ell (j_*M) = \sum_{i\geq 0} \partial^i_t \cdot \frac{1}{t}F_{\ell-i} V^{1}(j_*M)
\end{equation}
where $V^\bullet M$ is the $V$-filtration along $H$. Note also that when $\alpha>0$, we have 
\[
V^\alpha M=V^\alpha j_*M =V^\alpha j_!M.
\]
By the bistrictness of Hodge filtration and $V$-filtration~\cite{Saito_MHM}*{2.5} or~\cite{CD}*{Corollary 2.9}, the sequence
\[
 0 \to F_{\ell+1}V^\alpha \ker \tau \to F_{\ell+1}V^\alpha  i_*\Q^H_Z \to F_{\ell} V^\alpha  \cH^1_Z(\cO_X) \to F_{\ell+1}V^\alpha \coker \tau \to 0 
\]
is exact for any $\alpha\in \Q$ and $\ell \in \Z$. But when $\ell\leq k$ and $\alpha>0$, both $F_{\ell+1}V^\alpha \ker \tau$ and $F_{\ell+1}V^\alpha \coker \tau$ vanish because 
\[
F_{\ell+1}V^\alpha \ker\tau = ({j}_*j^*F_{\ell+1} \ker\tau) \cap V^\alpha \ker\tau 
\]
\text{and}
\[
F_{\ell+1} V^\alpha \coker \tau = ({j}_*j^*F_{\ell+1} \coker \tau) \cap V^\alpha \coker \tau
\] 
together with the fact that $F_{\ell+1}j^*\tau$ are isomorphisms. We deduce that 
\[
F_{\ell+1}V^\alpha(\tau) \colon F_{\ell+1}V^\alpha  i_*\Q^H_Z \to F_{\ell} V^\alpha  \cH^1_Z(\cO_X)
\]
is isomorphism for any $\ell\leq k$ and $\alpha>0$. This shows that $F_{\ell} j_!\tau$ and $F_{\ell+1}j_*\tau$ are isomorphisms by~\eqref{eq:ext} and we immediately see that 
\[
\gr^F_{k-d}\DR_X(j_!\tau) \quad \text{and} \quad \gr^F_{k-d}\DR_X(j_*\tau)
\] 
are both quasi-isomorphisms by~\eqref{eq:drconv}.
\end{proof}

It remains to prove that the canonical morphism
\begin{equation}\label{eq:acyc}
\gr^F_{k-d}\DR_X\left({j_{T}}_!i_*\Q^H_{Z\setminus T}[d]\right)\to\gr^F_{k-d}\DR_X\left({j}_! i_*\Q^H_{Z\setminus H}[d]\right) 
\end{equation}
is a quasi-isomorphism. Recall that $j'\colon X\setminus Z_\sing \to X$ is the open immersion and $i'\colon Z_\sing \to X$ is the closed immersion. Thanks to the distinguished triangle 
\[
{j_T}_!i_*\Q^H_{Z\setminus T} \to {j}_! i_*\Q^H_{Z\setminus H} \to {j}_!i'_*\Q^H_{Z_{\sing}} \xrightarrow{+1},
\]
obtained from applying $j_!j^!$ to 
\[
j'_!i_*\Q_{Z_\reg}^H \to  i_*\Q^H_Z \to \underbrace{i'_*i'^* i_*\Q_{Z}}_{=i'_*\Q_{Z_\sing}^H} \xrightarrow{+1},
\]
the assertion that~\eqref{eq:acyc} is an isomorphism is now equivalent to the acyclicity of $\gr^F_{k-d}\DR_X\left({j}_!i'_*\Q^H_{Z_{\sing}}\right)$. It is sufficient to show that $\underline\Omega^{d-k}_{Z_{\sing}\cap H}$ and $\underline\Omega^{d-k}_{Z_{\sing}}$ are both acyclic, due to the distinguished triangle
\[
\begin{aligned}
    \gr^F_{k-d}\DR_X\left({j}_!i'_*\Q^H_{Z_{\sing}}\right) \to & \\
\underbrace{\gr^F_{k-d}\DR_X\left(i'_*\Q^H_{Z_{\sing}}\right)}_{\cong\underline\Omega^{d-k}_{Z_{\sing}}[d-k]}& \to \underbrace{\gr^F_{k-d}\DR_X\left({i'_H}_*\Q^H_{Z_{\sing}\cap H}\right)}_{\cong\underline\Omega^{d-k}_{Z_{\sing}\cap H}[d-k]} \xrightarrow{+1},
\end{aligned}
\]
where $i'_H\colon Z_\sing\cap H\to X$ is the closed immersion. This is done by observing that $Z$ has $k$-Du Bois singularities and hence, $\dim Z_{\sing}\cap H  \leq \dim Z_{\sing}\leq d-2k-1 <d-k$ by Theorem~\ref{sing_locus_Du_Bois}.

\subsubsection{} 
The remaining is to prove that $\cH^0\bD_X\left(\gr^F_{-k}\DR_X(\beta)\right)$ is surjective. To this end, we make use of an auxiliary mixed Hodge module $M:=\cH^{-1}i_*i^*j_*j^*\Q_X^H[n]$. Note that $M$ is the same as $i_*\Q_{Z\setminus H}^H[d]$ over the non-characteristic locus $X'$, the complement of $Z_\sing$ and $(Z\cap H)_\sing$ in $X$, with respect to $j_*j^*\Q_X^H[n]$. Due to the adjunction $\id\to j_*j^*$, we obtain a factorization
\[
\beta\colon i_*\Q^H_Z[d] \xrightarrow{\gamma} M \xrightarrow{\delta} \underbrace{i_*\Q^H_{Z\setminus H}[d]}_{= j_*j^*M}
\] 
This reduces the surjectivity of $\cH^0\bD_X\left(\gr^F_{-k}\DR_X(\beta)\right)$ to the surjectivity of 
\[
\cH^0\bD_X\left(\gr^F_{-k}\DR_X(\gamma)\right) \quad \text{and} \quad \cH^0\bD_X\left(\gr^F_{-k}\DR_X(\delta)\right).
\]

\subsubsection{} 
We prove that $\cH^0\bD_X\left(\gr^F_{-k}\DR_X(\gamma)\right)$ is surjective. Let $i_H\colon H\to X$ be the closed immersion. Since $H$ is smooth, there is a short exact sequence of mixed Hodge modules
\begin{equation}\label{eq:xses}
0\to \Q^{H}_X[n] \to j_*j^*\Q^H_X[n] \to {i_{H}}_* \Q^H_H[n-1](-1) \to 0.
\end{equation}
Applying $i_*i^{*}$ to~\eqref{eq:xses} and then taking $(-1)$-th cohomology we get
\begin{equation}\label{eq:M}
0 \to i_*\Q_Z^H[d] \xrightarrow{\gamma}  M \to {i_D}_* \Q^H_D[d-1](-1) \xrightarrow{} 0,
\end{equation}
where $D:=Z\cap H$ and $i_D:D \to X$ is the closed embedding. Therefore, the dual of $\gamma$ is a surjection of mixed Hodge modules:
\[
\bD_X(\gamma)\colon \bD_X(M) \to \bD_X(i_*\Q_Z^H[d]).
\]
Because $\cH^0\bD_X(\gr^F_{-k}\DR_X(\gamma))=\cH^0\gr^F_{k}\DR_X(\bD_X(\gamma))$ by~\eqref{eq:dualprop}, and the top cohomological degree of the de Rham complex of a $\cD$-module is $0$; see~\eqref{eq:drconv}, we have completed this step.

\subsubsection{} 
We will prove in the remaining steps that 
\[
\gr^F_{-k}\DR_X(\delta) \colon \gr^F_{-k}\DR_{X}(M) \to \gr^F_{-k}\DR_X\left(i_*\Q^H_{Z\setminus H}[d]\right) 
\]
admits a left inverse in $\rm D^b_{\rm coh}(X)$. Here, by a \emph{left inverse} of a morphism $f\colon A\to B$, we mean a morphism $g\colon B\to C$ such that $g\circ f\colon A\to B\to C$ is an isomorphism. Clearly, the existence of a left inverse of $\gr^F_{-k}\DR_X(\delta)$ implies that $\cH^0\bD_X(\gr^F_{-k}\DR_X(\delta))$ is surjective by duality.

The observation is that
\[
\gr^F_{-k}\DR_X(j_*\tau)\colon \gr^F_{-k}\DR_X\left(i_*\Q^H_{Z\setminus H}[d]\right) \xrightarrow{\cong} \gr^F_{-k-1}\DR_X\left(j_*\cH^1_Z(\cO_X)\right). 
\]
is a quasi-isomorphism because its dual
\[
\bD_X(\gr^F_{-k}\DR_X(j_*\tau)) \cong \gr^F_{k}\DR_X(\bD_X(j_*\tau)) \cong \gr^F_{k-d}\DR_X\left(j_!\tau \right)
\]
is by Lemma~\ref{lem:ext} and $\bD_X(\tau)=\tau(d)$. Therefore, to construct a left inverse of $\gr^F_{-k}\DR_X(\delta)$, it suffices to construct a left inverse of the morphism:
\begin{equation}\label{eq_linver}
  \gr^F_{-k}\DR_X(j_*\tau\circ\delta) \colon  \gr^F_{-k}\DR_X(M)\to \gr^F_{-k-1}\DR_X\left(j_*\cH^1_Z(\cO_X)\right).
\end{equation}

\subsubsection{} 
We claim that there is a morphism 
\begin{equation}\label{eq:st}
\gr^F_{-k-1}\DR_X\left(j_*\cH^1_Z(\cO_X)\right) \to \gr^F_{-k-1}\DR_X\left(j_!\cH^1_Z(\cO_X)\right)\otimes_{\cO_X} \cO_X(H),
\end{equation}
and it will be shown in the next step that it is a left inverse of~\eqref{eq_linver}. 

Take a log resolution $f:Y\to X$ of the pair $(X,Z+H)$ such that $f$ is isomorphic over the simple normal crossing locus $X'$ of $(X,Z+H)$; such log resolution always exist~\cite{SMMP}*{Theorem 10.45}. In our case, the simple normal crossing locus $X'$ is the complement of $Z_\sing$ and $(Z\cap H)_\sing$ in $X$. Set $\widetilde Z:={(f^*Z)}_{\rm red}$, $\widetilde H:={(f^*H)}_{\rm red}$ and $G:=(\widetilde Z+\widetilde H)_{\rm red}$. Taking $\gr^F_{-k-1}\DR_X$ of the following short exact sequence of mixed Hodge modules
\[
0\to {j_*j^*\cO_X} \to {j_*j^*\cO_X(*Z)} \to j_*j^*\cH^1_Z(\cO_X) \to 0
\]
gives a distinguished triangle by~\eqref{eq:logres}:
\begin{equation}\label{eq:ses*}
\begin{aligned}
     \Omega^{k+1}_X(\log H) \to & \bR f_*\Omega^{k+1}_Y(\log G)  \\
    &  \to\gr^F_{-k-1}\DR_X\left(j_*\cH^1_Z(\cO_X)\right)[k-d] \xrightarrow{+1}.
\end{aligned}
\end{equation}
 
We also make use of another short exact sequence of mixed Hodge modules:
\begin{equation}\label{eq:!*}
0\to {j_!j^!\cO_X} \to {j_!j^!\cO_X(*Z)} \to j_!j^!\cH^1_Z(\cO_X) \to 0.
\end{equation}
Let $U$ be the complement of $Z+H$ in $X$. Let $\tilde j\colon Y\setminus \widetilde H\to Y$, $\tilde h \colon U\to Y\setminus \widetilde H$ and $h\colon U\to X\setminus H$ be the open immersions. Their relations are summarized in the following Cartesian diagram.
\[
\begin{tikzcd}
U \arrow{r}{\tilde h}\arrow[d, equal] & Y\setminus \widetilde H  \arrow{r}{\tilde j} \arrow{d} &   Y \arrow{d}{f} \\
U \arrow{r}{h} & X\setminus H \arrow{r}{j} & X
\end{tikzcd}
\]
Because $f$ is proper, it follows from 
\[
{j_!j^!\cO_X(*Z)}=j_!h_*\cO_U=f_+\tilde{j}_!\tilde{h}_*\cO_U=f_+\tilde{j}_!\tilde{j}^!\cO_Y(* \tilde Z)=f_+\tilde{j}_!\tilde{h}_*\cO_U
\] 
and Lemma~\ref{lem:logcomapre} below that we can compute $\gr^F_{-k-1}\DR_X\left({j_!j^!\cO_X(*Z)}\right)$ by
\[
\bR f_*\left(\Omega^{k+1}_Y\left(\log G\right)(-\widetilde H)\right)[n-k-1] \xrightarrow{\cong} \gr^F_{-k-1}\DR_X\left(f_+\tilde{j}_!\tilde{j}^!\cO_Y\left(* \widetilde Z\right)\right)
\]
because of~\eqref{eq:direct}. Then applying $\gr^F_{-k-1}\DR_X$ to~\eqref{eq:!*} gives another distinguished triangle:
\begin{equation}\label{eq:ses!}
\begin{aligned}
     \Omega^{k+1}_X(\log H)(-H) \to & \bR f_*\left(\Omega^{k+1}_Y\left(\log G\right)(-\widetilde H)\right) \\
    & \to \gr^F_{-k-1}\DR_X\left(j_!\cH^1_Z(\cO_X)\right)[k-d] \xrightarrow{+1}. 
\end{aligned}
\end{equation}

Since $f^*H-\widetilde H$ is effective, there is a canonical map
\[
\theta \colon \Omega^{k+1}_Y\left(\log G \right) \to \Omega^{k+1}_Y\left(\log G\right)\left(f^*H-\widetilde H\right).
\] 
Combined with the projection formula, we then have the following commutative diagram. 
\[
\begin{tikzcd}
\Omega^{k+1}_X(\log H) \arrow{r}\arrow[d, "="] & \bR f_*\Omega^{k+1}_Y(\log G)\arrow{d}{\bR f_*\theta} \\
\Omega^{k+1}_X(\log H)(-H) \otimes\cO_X(H) \arrow{r} & \bR f_*\left(\Omega^{k+1}_Y\left(\log G \right)(-\widetilde H)\right) \otimes \cO_X(H)
\end{tikzcd}
\]
Comparing~\eqref{eq:ses*} with~\eqref{eq:ses!}, the above diagram indicates that there is a (non-canonical) morphism by the axiom $\rm TR3$ of triangulated categories
\begin{equation}\label{eq_4.1}
    \gr^F_{-k-1}\DR_X\left(j_*\cH^1_Z(\cO_X)\right) \to \gr^F_{-k-1}\DR_X\left(j_!\cH^1_Z(\cO_X)\right)\otimes_{\cO_X} \cO_X(H)
\end{equation}
in $\rm D^b_{\rm coh}(X)$, which is isomorphic over the simple normal crossing locus $X'$. This is exactly the morphism we are after.

\subsubsection{}\label{subsubsec:7}
We use the idea in~\cite{Sch}*{Theorem 5.1} to prove that~\eqref{eq_4.1} is a left inverse of~\eqref{eq_linver}; i.e. the composition of the morphisms
\[
\begin{aligned}
     \phi\colon  & \gr^F_{-k}\DR_{X}(M)  \xrightarrow{}  \gr^F_{-k}\DR_X\left(i_*\Q^H_{Z\setminus H}[d]\right) \\
\xrightarrow{\cong} & \gr^F_{-k-1}\DR_X(j_*\cH^1_Z(\cO_X)) \to \gr^F_{-k-1}\DR_X\left(j_!\cH^1_Z(\cO_X)\right)\otimes_{\cO_X} \cO_X(H)
\end{aligned}
\]
is an isomorphism in $\rm D^b_{\rm coh}(X)$.

We first argue that the source and target of $\phi$ are supported in cohomological degree $k-d$. Taking $\gr^F_{-k}\DR_X$ of~\eqref{eq:M}:
\[
0 \to i_*\Q_Z^H[d] \xrightarrow{\gamma}  M \to {i_D}_* \Q^H_D[d-1](-1) \xrightarrow{} 0
\] 
gives a distinguished triangle:
\begin{equation}\label{eq:logses}
 \Omega_Z^k \to \gr^F_{-k}\DR_X(M)[k-d] \to \Omega_{D}^{k-1} \xrightarrow{+1}
\end{equation}
by~\eqref{eq:defdb} because both $Z$ and $D=Z\cap H$ have $k$-Du Bois singularities. 
Therefore, $\gr^F_{-k}\DR_X(M)\cong\cF[d-k]$ for a coherent $\cO_Z$-module $\cF$. Moreover, if we put $W:=Z\cap X'=Z_{\reg}\setminus D_{\sing}$, then
\[
\cF|_W=\Omega^k_W(\log D_W)
\] 
where $D_W:=D\cap W$, as $M\vert_{X'}=i_+\cO_Z(*D)\vert_{X'}$.

Next, we observe that the morphism
\begin{equation}\label{eq_ax}
    \gr^F_{-k}\DR_X(j_!\tau) \colon \gr^F_{-k}\DR_X\left(j_!i_*\Q^H_Z[d]\right)  \to \gr^F_{-k-1}\DR_X\left(j_!\cH^1_Z(\cO_X)\right)
\end{equation}
is a quasi-isomorphism as its dual 
$\bD_X(\gr^F_{-k}\DR_X(j_!\tau)\cong \gr^F_{k-d}\DR_X(j_*\tau)$
is by Lemma~\ref{lem:ext}, recalling that $\bD_X(\tau)=\tau(d)$. Then by rotating $\gr^F_{-k}\DR_X$ of the distinguished triangle
\[
{i_D}_*\Q^H_D[d-1] \to j_!i_*\Q^H_Z[d] \to i_*\Q^H_Z[d] \xrightarrow{+1}
\]
we get another distinguished triangle
\begin{equation}\label{eq:!dtrig}
\gr^F_{-k}\DR_X\left(j_!i_*\Q^H_Z[d]\right)[k-d] \to \Omega^k_Z \to \Omega^k_D \xrightarrow{+1}
\end{equation}
again thanks to the fact that both $D$ and $Z$ have $k$-Du Bois singularities.
This implies, as $\Omega^k_Z \twoheadrightarrow \Omega_D^k$, that 
\[
\gr^F_{-k-1}\DR_X\left(j_!\cH^1_Z(\cO_X)\right)\cong \gr^F_{-k}\DR_X\left(j_!i_*\Q^H_Z[d]\right)\cong \cE[d-k]
\] 
where the $\cO_Z$-module $\cE$ is the kernel of $\Omega^k_Z \twoheadrightarrow \Omega_D^k$. Moreover, we also see that $\cE|_W = \Omega^k_W(\log D_W)(-D_W)$. 

\subsubsection{}
We have reduced the proof of $\phi$ is quasi-isomorphism to that 
\[
\cH^{k-d}(\phi) \colon \cF\to \cE\otimes_{\cO_X}\cO_X(H)=\cE(D)
\] 
is an isomorphism as $\cO_Z$-modules. Note that by the discussion in~\ref{subsubsec:7}, $\cH^{k-d}(\phi)\vert_W$ is isomorphic because $\delta$ and~\eqref{eq_4.1} are isomorphic over $X'$; indeed, it even identifies $\cF\vert _W$ and $\cE(D)\vert _W$ because both are equal to $\Omega^k_W(\log D_W)$. Hence, by the adjunction $\id \to {j_W}_*j^*_W$ for $\cO_Z$-modules, we get the following commutative diagram.
\[
\begin{tikzcd}
\cF \arrow{r}{\cH^{k-d}(\phi)} \arrow{d}	&	\cE(D) \arrow{d}  \\
 {j_W}_*(\cF\vert_W) \arrow[r,equal] & {j_W}_*(\cE(D)\vert_W) 
\end{tikzcd}
\]
Thus, the proof will be concluded if we can show that $\cF={j_W}_*(\cF\vert_W)$ and that $\cE(D)$ is $\cO_Z$-torsion free. 

Clearly, the $\cO_Z$-module $\cE(D)$ is torsion free as $\cE$ is a subsheaf of the torsion free $\cO_Z$-module $\Omega^k_Z$. Recall that for a local complete intersection $Y$ the $\cO_Y$-module $\Omega_Y^p$ is even reflexive. 

Lastly, because the codimension of $D\setminus W$ in $D$ is at least $2$ by Theorem~\ref{sing_locus_Du_Bois}, we have the following commutative diagram, 
\[
\begin{tikzcd}[column sep=small]
0\arrow{r} & \Omega_Z^k \arrow{r}\arrow[d,"="] &  \cF \arrow{r}\arrow{d} & \Omega_{D}^{k-1} \arrow{r}\arrow[d,"="] &  0 \\
0\arrow{r} & {j_W}_*(\Omega_Z^k\vert_W) \arrow{r} &  {j_W}_*(\cF\vert_W) \arrow{r} & {j_W}_*(\Omega_{D}^{k-1}\vert_W) & 
\end{tikzcd}
\]
obtained by applying the adjunction $\id \to {j_W}_*j_W^*$ for $\cO_Z$-modules to the short exact sequence~\eqref{eq:logses}. The two outermost vertical maps are identity because $\Omega^k_Z$ and $\Omega^{k-1}_D$ are reflexive as $\cO_Z$-module and $\cO_D$-module, respectively. We have concluded the proof because the $5$-lemma implies that $\cF={j_W}_*(\cF\vert_W)$.  \hfill $\square$

\subsubsection{Log comparison}
The following can be essentially proved as in~\cite{Saito_MHM}*{3.11} via compatible $V$-filtrations on $\cD$-modules of normal crossing type; see also~\cite{Wei}. We sketch a proof of a simpler (but sufficient for application) statement for the reader's convenience.

\begin{lem} \label{lem:logcomapre}
    Let $X$ be a smooth complex algebraic variety of dimension $n$. Let $D$ and $E$ be two reduced effective divisors on $X$ such that the divisor $D+E$ has simple normal crossing support. Denote by $j\colon X\setminus E \to X$ be the open immersion. Then, for every $k\in \Z$, there is a natural quasi-isomorphism:
    \[
     \Omega_X^{k}(\log D+E)(-E)[n-k] \xrightarrow{\cong} \gr^F_{-k}\DR_X (j_!j^!\cO_X(*D)). 
    \]
\end{lem}
\begin{proof}
    Deleting the common irreducible components from $D$, we can assume that $D$ and $E$ have no common components. Put $G:=D+E$.

    We argue inductively on the number of the irreducible components of $E$. When $E$ has $0$ irreducible components, i.e. $E$ is empty, the assertion was proved in~\cite{deligne06} and~\cite{Saito_MHM}*{3.11}. 

    Let $H$ be an irreducible component of $E$ and $E':=E-H$. Let $i_H\colon H\to X$ be the closed immersion and $j'\colon X\setminus E'\to X$ be the open immersion. Let $\mathcal N:=j'_!j'^! \cO_X(*D)$. By the base change formula, 
    \[
    (\cN_H,F):=(\cH^{-1}i^+_H\cN,F)={j_H'}_!{j_H'}^!\cO_H(*D_H)
    \] 
    as mixed Hodge modules and $\cH^0i^+_H\cN$ vanishes, where $j_H'\colon H\setminus E'\to H$ and $D_H:=D\cap H$. We then have a short exact sequence of mixed Hodge modules 
    \begin{equation} \label{eq:ind}
        0\to {i_H}_+\cN_H \to \underbrace{j_!j^!\cO_X(*D)}_{=\cN(!H)} \to \cN \to 0.
    \end{equation}
    An inductive argument, together with the exact sequence obtained by the strictness of the Hodge filtration 
    \begin{equation}
        0\to \underbrace{F_0{i_H}_+\cN_H}_{=\omega_{H/X}\otimes F_{-1}\cN_H} \to F_0j_!j^!\cO_X(*D) \to F_0\cN \to 0,
    \end{equation}
    shows that the lowest non-zero piece of the Hodge filtration is $F_0j_!j^!\cO_X(*D)$ and 
    \[
    F_0j_!j^!\cO_X(*D)\cong F_0\cO_X(*D)=\cO_X(D).
    \] 
    Then we see that there is an inclusion 
    \[
    \Omega^k_X(\log G)(-E) \hookrightarrow \Omega^k_X(D)\cong \Omega^k_X\otimes F_0j_!j^!\cO_X(*D). 
    \]
    This can be extended to an inclusion of complexes
    \begin{equation}\label{eq:login}
         \Omega^k_X(\log G)(-E)[n-k] \hookrightarrow    \gr^F_{-k}\DR_X(j_!j^!\cO_X(*D))
    \end{equation}
    as the right-hand side is:
    \[
        0\to \underbrace{\Omega_X^k\otimes F_0j_!j^!\cO_X(*D)}_{\mathrm{degree}\, k-n} \to \cdots \to \underbrace{\omega_X\otimes \gr^F_{n-k}j_!j^!\cO_X(*D)}_{\mathrm{degree }\, 0} \to 0.
    \]

    Consider the distinguished triangle obtained from rotating $\gr^F_{-k}\DR_X$ of~\eqref{eq:ind}:
    \begin{equation}\label{eq:logdist}
    \gr^F_{-k} \DR_X(j_!j^!\cO_X(*D)) \to \gr^F_{-k} \DR_X(\cN) \xrightarrow{\varepsilon} \underbrace{\gr^F_{-k} \DR_H(\cN_H)[1]}_{\cong \gr^F_{-k} \DR_X({i_H}_+\cN_H)[1]} \xrightarrow{+1},
    \end{equation}
    where $\varepsilon$ is induced by pull-back of K\"ahler differentials. Then~\eqref{eq:login} gives a morphism from the  distinguished triangle 
    \[
    \begin{aligned}
        \Omega_X^{k}(\log G)(-E)[n-k] \to  & \\
        \Omega_X^{k}(\log G')(-E')&[n-k]  \xrightarrow{\varepsilon} 
     \Omega_H^k(\log G'_H)(-E'_H)[n-k] \xrightarrow{+1}
    \end{aligned}
    \]
    to~\eqref{eq:logdist}, where $G':=D+E'$, $E'_H:=E'\cap H$ and $G'_H:=G'\cap H$. Now because of the induction hypothesis on $\cN$ and $\cN_H$, an application of the $5$-lemma gives
    \[
     \Omega_X^{k}(\log G)(-E)[n-k] \to \gr^F_{-k} \DR_X(j_!j^!\cO_X(*D))
    \]
    is an isomorphism. 
\end{proof}

\begin{rmk}
    Indeed, there is a natural filtered quasi-isomorphism  
    \[
    \Omega^{n+\bullet}_X(\log D+E)(-E) \to (\DR_X(j_!j^!\cO_X(*D)),F).
    \]
    To prove this, it suffices to upgrade~\eqref{eq:login} into
    \[
         \sigma_{\geq k-n}\Omega^{\bullet+n}_X(\log G)(-E) \to    F_{-k}\DR_X(j_!j^!\cO_X(*D))
    \] 
    an inclusion of a subcomplex, where $\sigma_{\geq p}$ is the truncation at degree $p$. This can be checked using the description $j_!j^!\cO_X(*D)=\cD_X \otimes _{\cD_X(\log G)} \cO_X(D)$, from which we can see the $\cD$-module structure on $j_!j^!\cO_X(*D)$ clearly. 
\end{rmk}

We conclude the paper by

\subsection{Proof of Theorem~\ref{thm:inratbd}}
    Since the statement is local on $Z$, we can assume that there is a closed immersion  $i\colon Z\to X$ into an irreducible and smooth complex algebraic variety $X$ such that $Z$ has pure codimension $r$ in $X$. Then there is a hypersurface $H$ in $X$ containing no irreducible component of $Z$ such that $D=Z\cap H$. By taking $c=k+r$ in Theorem~\ref{thm:inme} and applying Theorem~\ref{thm:medbrat}, we conclude the proof. \hfill $\square$

\end{document}